\documentclass{amsart}

\usepackage{amssymb}
\usepackage{color}

\theoremstyle{plain}

\numberwithin{equation}{section}

\newtheorem{teo}{Theorem}

\newtheorem{cor}{Corollary}

\newtheorem{lemma}{Lemma}
\newtheorem{cora}[lemma]{Corollary}

\theoremstyle{definition}
\newtheorem*{notation}{Notation}
\newtheorem{remark}{Remark}

\newcommand{\Type}{\operatorname{Type}}

\theoremstyle{definition}

\usepackage{fancyhdr,ifthen}
\begin{document}	

\title{ Symmetric groups and conjugacy classes}

\author{Edith Adan-Bante and  Helena Verrill}

\address{University of Southern Mississippi Gulf Coast, 730 East Beach Boulevard,
 Long Beach MS 39560} 
\email{EdithAdan@illinoisalumni.org}

\address{Department of Mathematics, Louisiana State University, Baton Rouge, LA 70803-4918}
\email{verrill@math.lsu.edu}

\subjclass{20b30}

\begin{abstract} Let $S_n$ be the symmetric group on $n$-letters. Fix $n>5$.
Given any nontrivial 
$\alpha,\beta\in S_n$, we prove that the product $\alpha^{S_n}\beta^{S_n}$
of the conjugacy classes $\alpha^{S_n}$ and $\beta^{S_n}$ is never a conjugacy 
class. Furthermore, if $n$ is not even and $n$ is not a multiple of three, then 
$\alpha^{S_n}\beta^{S_n}$ is the union of at least three distinct conjugacy 
classes. We also describe the elements $\alpha,\beta\in S_n$ in the case when 
$\alpha^{S_n}\beta^{S_n}$ is the union of exactly two distinct conjugacy 
classes.
\end{abstract}
\maketitle

\begin{section}{Introduction}

Let $G$ be a finite group, $a\in G$  and $a^G=\{a^g\mid g\in G\}$ 
be the conjugacy class of $a$ in $G$.
Let $X$ be a $G$-invariant subset of $G$, i.e. 
$X^g=\{x^g\mid x\in X\}=X$ for all $g\in G$.
  Then $X$ can be expressed as a union of 
  $n$ distinct conjugacy classes of $G$, for some integer $n>0$. Set
 $\eta(X)=n$.
 
For any
$a,b\in G$, the product
 $a^G b^G=\{xy\mid x\in a^G, y\in b^G\}$
is a $G$-invariant set.
In this note we explore $\eta(a^G b^G)$ when
$G$ is the symmetric group $S_n$ on $n$-letters, and $a^G$,
$b^G$ are conjugacy classes of $S_n$.  We denote the
identity of any group by $e$.

Arad and Herzog
conjectured that the product of two nontrivial 
conjugacy classes in a finite simple
nonabelian  group is never a conjugacy class \cite{arad},
i.e., if $a,b\not=e$, then  $\eta(a^Gb^G)\not=1$.
This has been proved in some cases \cite{arad}, 
in particular, it has been proved for the alternating group $A_n$,  i.e.
if $n\ge 5$ and $\alpha,\beta\in A_n\setminus\{e\}$, then
$\eta(\alpha^{A_n}\beta^{A_n})\ge 2$.

In this note, we show that
the symmetric group behaves similarly, and we give an explicit description of
the minimum possible value of $\eta$.
 \begin{teo}
 Let $S_n$ be the symmetric group on $n$-letters,
$n>5$, and $\alpha,\beta\in S_n
\setminus\{e\}$. 
Then $\eta(\alpha^{S_n}\beta^{S_n})\ge 2$, and
 if $\eta(\alpha^{S_n}\beta^{S_n})=2$ then either
$\alpha$ or $\beta$ is a fixed point free permutation. Assume that
 $\alpha$ is fixed point free. Then 
  one of the following
 holds
 
 i) $n$ is even,  $\alpha$  is the product of $\frac{n}{2}$ disjoint
 transpositions and
 $\beta$ is either a transposition or a $3$-cycle.

 ii) $n$ is a multiple of $3$,  $\alpha$ is the  product
 of $\frac{n}{3}$ disjoint 3-cycles and
 $\beta$ is a transposition.
 \end{teo}
 
 Since for any group $G$ and any $a,b\in G$, we have
 $a^G b^G= b^G a^G$ (see Lemma \ref{interchange}), 
Theorem A describes $\alpha$ and $\beta$
 when $\eta(\alpha^{S_n}\beta^{S_n})=2$.
 \begin{cor} Fix $n>5$.
 Let $m$  be the least integer in 
 $\{\eta(\alpha^{S_n}\beta^{S_n})\mid \alpha,\beta \in S_n\setminus\{e\}\}$. Then
\begin{enumerate} 
\item[ i)]  $m=2$ if $n$ is divisible by $2$ or $3$,
 \item[ ii)] $m=3$ otherwise.
\end{enumerate}
  \end{cor}

\begin{remark}
For $n=5$ we have
$(1\ 2)^{S_5}
(1\ 2\ 3\ 4\ 5)^{S_5}=(1\ 2\ 3\ 4)^{S_5}\cup
((1\ 2\ 3)(4\ 5))^{S_5}$.  This describes,
up to conjugation and ordering,
the only pair of elements $\alpha,\beta \in S_5$ with $\eta(\alpha^{S_5}\beta^{S_5})=2$.
\end{remark}

 As for the maximum possible value of 
$\eta(a^Gb^G)$,
John Thompson conjectured that given any finite nonabelian simple group $G$, there exists
a conjugacy class $C$ such that $C^2=G$ [see \cite{kappe}]. The conjecture has been 
proved for the alternating group $A_n$ with $n\geq 5$
[see \cite{hsii}]. Since given any $\alpha,\beta\in S_n$,
either $\alpha^{S_n} \beta^{S_n}\subseteq A_n$ or $\alpha^{S_n} \beta^{S_n}\subseteq S_n\setminus A_n$, it follows then that there exists a conjugacy class $C$ in $S_n$
such that $\eta(C^2)$ is the number of conjugacy classes of $S_n$ in $A_n$ and that
is the largest possible value for $\eta(\alpha^{S_n} \beta^{S_n})$. 
See \cite{homogen}, \cite{edith2}, \cite{derived}, \cite{dade} for 
examples of recent developments
in products of conjugacy classes. The products of conjugacy classes of symmetric groups
have been studied extensively, for instance  in
\cite{bedard}, \cite{farahathigman} \cite{goupil1}, \cite{goupil2} and \cite{jackson}.

The results of this paper were discovered by experimentation, using the
computer algebra package  MAGMA \cite{magma}.

{\bf Acknowledgment.} The first author would like to thank FEMA for providing her
 with temporary housing in the  aftermath of hurricane Katrina. 

The second author is partially supported by NSF grant  DMS-0501318
and Louisiana Board of Regents grant LEQSF-(2004-7) RD-A-16.

\end{section}

\begin{section}{Notation}
Our notation makes use of the following 
very well known result. 
\begin{lemma}
\label{lemma1}
Let $\alpha \in S_n$. Then

a) $\alpha$ can be written as a product of disjoint cycles.

b) $\alpha^{S_n}$ is the set of all permutations of $S_n$ with the same
cycle structure as $\alpha$.
\end{lemma}

\begin{cora}
If $\alpha,\beta\in S_n$
\begin{enumerate}
\item[a)]
have different cycle structures, then they are not conjugate.
\item[b)]
have different numbers of fixed points,
then they are not conjugate.
\end{enumerate}
\end{cora}

\begin{notation}
We follow standard conventions, but repeat this
for clarity.
\begin{itemize}
\item
For a positive integer $n$,
$S_n$ denotes the symmetric group on $n$ objects, which we
identify with the set of bijective endomorphisms of the set
$\{1,\dots,n\}$.
\item
For a group $G$ and
 $a,b\in G$, we write $a\sim b$ if for some $g\in G$ we have
$a^g=b$, where $a^g=g^{-1}ag$.
\item
For $\sigma\in S_n$, the cycle structure of $\sigma$ is
the multiset of the lengths of all the disjoint cycles comprising $\sigma$.
We denote the cycle structure of $\sigma$ by $\Type(\sigma)$.
For example,
$\Type((1\ 3)(2\ 5\ 6)(4)(7) )=\{1,1,2,3\}$.  By Lemma~\ref{lemma1},
$\Type(\sigma)$ is well defined, and 
$\sigma\sim\tau\iff
\Type(\sigma)=\Type(\sigma)$.
\end{itemize}
\end{notation}

\end{section}
\begin{section}{Proofs}

\begin{lemma}
\label{interchange}
Let $G$ be a finite group, and let $a^G$ and $b^G$ be conjugacy classes of $G$. Then $a^G b^G=b^G a^G$.
In particular, $a^G b^G=(ab)^G$ if and only if $b^G a^G=(ba)^G$.
\end{lemma}
\begin{proof}
Observe that $ba^g= (a^g)^{b^{-1}}b$, so $b^G a^G\subseteq  a^G b^G$.
Similarly,
 $a^G b^G\subseteq b^G a^G$.
\end{proof}

\begin{lemma}\label{bypieces}
Let $\alpha,\beta \in S_r$ for $r<n$.  We may consider 
$\alpha$ and $\beta$ as elements of $S_n$ by defining their action to be
trivial on $r+1,\dots,n$.
Suppose that $\alpha_1,\beta_1\in S_n$ fix $1,\dots,r$.
Then
$\eta(\alpha^{S_r}\beta^{S_r})\leq 
\eta((\alpha\alpha_1)^{S_n}(\beta\beta_1)^{S_n})$.
\end{lemma}
\begin{proof}
If $\eta(\alpha^{S_r}\beta^{S_r})=k$, by Lemma \ref{lemma1} we have that there exist
$\sigma_1, \ldots, \sigma_k\in S_r$ with
$\Type(\alpha^{\sigma_i}\beta)\not=
\Type(\alpha^{\sigma_j}\beta)$ for $i\not=j$.
Since 
$\Type(\alpha^{\sigma_i}\alpha_1\beta\beta_1)
=\Type(\alpha^{\sigma_i}\beta)\cup\Type(\alpha_1\beta_1)$,
$\alpha^{\sigma_i}\alpha_1\beta\beta_1$ are pairwise non-conjugate,
and the result follows.
\end{proof}

\begin{lemma}
\label{lem:alpha_cong_nowhere=beta}
For $\alpha,\beta\in S_n$, 
$n\ge 5$,
if $\alpha$ has $r$ fixed points,
and $\beta$ has $s$ fixed points, then provided $r+s\le n$,
there exists $\sigma\in S_n$
with
 $\alpha^\sigma(i)\not=\beta(i)$ for
$i=1,\dots,n$.
\end{lemma}
\begin{proof}
For a positive integer $k\le n$, we  inductively define
$\sigma_k\in S_n$ such that
$\alpha^{\sigma_k}(i)\not=\beta(i)$ for
$1\le i\le k$.  
We take $\sigma_0=e$.
Suppose $\sigma_{k-1}$ has been defined for some $k> 0$.
We define $\sigma_{k}$  as follows:
\begin{enumerate}
\item[case 1:]
If
$\alpha^{\sigma_{k-1}}(k)\not=\beta(k)$
then set $\sigma_k=\sigma_{k-1}$.
\item[case 2:]
Suppose $\alpha^{\sigma_{k-1}}(k)=\beta(k)=k$.
Let $A$ be the subset of $\{1,\dots,n\}$ fixed by $\alpha^{\alpha_k}$,
and let $B$ be the subset fixed by $\beta$.
Since $|A|+|B|=r+s\le n$, and since $k\in A\cap B$, we must
have some $t\in\{1,\dots,n\}\setminus(A\cup B)$, and this $t$ satisfies
$\alpha^{\sigma_{k-1}}(t)\not=t$ and
$\beta(t)\not=t$.
Set
$\sigma_k=\sigma_{k-1}(k\ t)$.
This satisfies the required condition, since
$\alpha^{\sigma_k}(i)=\alpha^{\sigma_{k-1}}(i)$ unless
$i=t $ or $i=k$, and
$\alpha^{\sigma_k}(t)=k\not=\beta(t)$
(because $\beta(i)=k$ only if $i=k$, and $k\not=t$)
  and
$\alpha^{\sigma_k}(k)=\alpha^{\sigma_{k-1}}(t)\not=\beta(k)=k
$
(because $\alpha^{\sigma_{k-1}}(k)=k$, so 
$\alpha^{\sigma_{k-1}}(t)$ has some other value).
\item[case 3:]
$\alpha^{\sigma_{k-1}}(k)=\beta(k)\not=k$.
In this case, since $n\ge 5$, there is some $t\in\{1,\dots,n\}$ with 
$t\notin
S:=
\{\alpha^{\sigma_{k-1}}
\beta^{-1}(k),
\beta(\alpha^{\sigma_{k-1}})^{-1}(k),\alpha^{\sigma_{k-1}}(k),k\}$
(these are labeled in the top row of the picture
below).
Set $\sigma_k=\sigma_{k-1}(k\ t)$.
We have
$\alpha^{\sigma_k}(i)=\alpha^{\sigma_{k-1}}(i)$ unless
$i\in R:=\{k, t, u:=(\alpha^{\sigma_{k-1}})^{-1}(k),
v:=(\alpha^{\sigma_{k-1}})^{-1}(t)\}$. So, for $j\le k$ not in $R$,
 we have $\alpha^{\sigma_k}(j)\not=\beta(j)$.
 We  now show that for  $i\in R$ we also have
$\alpha^{\sigma_k}(i)\not=\beta(i)$.
For $i=k$, we have
$\alpha^{\sigma_k}(k)=\alpha^{\sigma_{k-1}}(t)\not=\beta(k)$,
since otherwise
 $\alpha^{\sigma_{k-1}}(t)=\beta(k)=\alpha^{\sigma_{k-1}}(k)$ implies
$k=t$, contradicting $t\notin S$.
For $i=t$, from the assumption $t\not=\alpha^{\sigma_{k-1}}(k)$ it follows that 
$\alpha^{\sigma_k}(t)=\alpha^{\sigma_{k-1}}(k)$.
We have $\alpha^{\sigma_{k-1}}(k)\not=\beta(t)$,
since otherwise $t=
\beta^{-1}\alpha^{\sigma_{k-1}}(k)=
\beta^{-1}\beta(k)=k$, contradicting $t\not=k$.
If $i=u$, then
$\alpha^{\sigma_k}(u)=t\not=\beta(u)$ because otherwise
$t=\beta(u)=\beta(\alpha^{\sigma_{k-1}})^{-1}(k)$, contradicting $t\notin S$.
For $i=v$, we have
$\alpha^{\sigma_k}(v)=k\not=\beta(v)$ because 
otherwise
$k=\beta(v)=
\beta(\alpha^{\sigma_{k-1}})^{-1}(t)$,
so $t=\alpha^{\sigma_{k-1}}\beta^{-1}(k)$, contradicting $t\notin S$.
\end{enumerate}
Proceeding in this
way, we eventually obtain $\sigma=\sigma_n$ as required.

$$
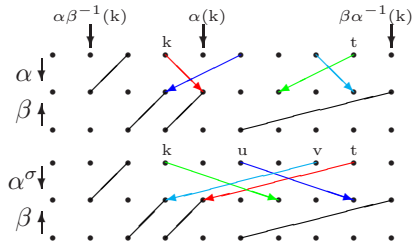
\setlength\unitlength{0.5cm}
\begin{picture}(22,2.8)(-1,0)
\put(12,1){\parbox{2in}{Picture corresponding
 to case 3 of Lemma~\ref{lem:alpha_cong_nowhere=beta}.
This shows a possible choice of $t$.\\
For simplicity we write $\alpha$ for $\alpha^{\sigma_{k-1}}$
and $\sigma$ for $(k\ t)$.
}}
\put(-1,0.3){$\beta$}
\put(-1,1.3){$\alpha$}
\put(-0.3,1.9){\vector(0,-1){0.6}}
\put(-0.3,0.1){\vector(0,1){0.6}}
\put(-0.08,2.2){\tiny}
\put(0,2.9){\tiny $\alpha\beta^{-1}$(k)}
\put(7.9,2.2){\tiny t}
\put(2.9,2.2){\tiny k}
\put(3.6,2.9){\tiny $\alpha$(k)}
\put(7.6,2.9){\tiny $\beta\alpha^{-1}$(k)}
%
\put(1,2.8){\vector(0,-1){0.6}}
\put(4,2.8){\vector(0,-1){0.6}}
\put(9,2.8){\vector(0,-1){0.6}}
\put(0,0){\circle*{0.15}}\put(1,0){\circle*{0.15}}
\put(2,0){\circle*{0.15}}\put(3,0){\circle*{0.15}}
\put(4,0){\circle*{0.15}}\put(5,0){\circle*{0.15}}
\put(6,0){\circle*{0.15}}\put(7,0){\circle*{0.15}}
\put(8,0){\circle*{0.15}}\put(9,0){\circle*{0.15}}
\put(0,1){\circle*{0.15}}\put(1,1){\circle*{0.15}}
\put(2,1){\circle*{0.15}}\put(3,1){\circle*{0.15}}
\put(4,1){\circle*{0.15}}\put(5,1){\circle*{0.15}}
\put(6,1){\circle*{0.15}}\put(7,1){\circle*{0.15}}
\put(8,1){\circle*{0.15}}\put(9,1){\circle*{0.15}}
\put(0,2){\circle*{0.15}}\put(1,2){\circle*{0.15}}
\put(2,2){\circle*{0.15}}\put(3,2){\circle*{0.15}}
\put(4,2){\circle*{0.15}}\put(5,2){\circle*{0.15}}
\put(6,2){\circle*{0.15}}\put(7,2){\circle*{0.15}}
\put(8,2){\circle*{0.15}}\put(9,2){\circle*{0.15}}
\color{red}
\put(3,2){\vector(1,-1){1}}
\color{blue}
\put(5,2){\vector(-2,-1){2}}
\color{green}
\put(8,2){\vector(-2,-1){2}}
\color{cyan}
\put(7,2){\vector(1,-1){1}}
\color{black}
\put(4,1){\line(-1,-1){1}}
\put(3,1){\line(-1,-1){1}}
\put(1,1){\line(1,1){1}}
\put(5,0){\line(4,1){4}}
\end{picture} 
$$
$$
\setlength\unitlength{0.5cm}
\begin{picture}(22,2)(-1,0.5)
\put(12,0.8){\parbox{2in}{The second picture shows the result of 
conjugating $\alpha$ by $(k\ t)$}}
\put(-1,0.3){$\beta$}
\put(-1.2,1.3){$\alpha^\sigma$}
\put(-0.3,1.9){\vector(0,-1){0.6}}
\put(-0.3,0.1){\vector(0,1){0.6}}
\put(7.9,2.2){\tiny t}
\put(2.9,2.2){\tiny k}
\put(4.9,2.2){\tiny u}
\put(6.9,2.2){\tiny v}
\put(0,0){\circle*{0.15}}\put(1,0){\circle*{0.15}}
\put(2,0){\circle*{0.15}}\put(3,0){\circle*{0.15}}
\put(4,0){\circle*{0.15}}\put(5,0){\circle*{0.15}}
\put(6,0){\circle*{0.15}}\put(7,0){\circle*{0.15}}
\put(8,0){\circle*{0.15}}\put(9,0){\circle*{0.15}}
\put(0,1){\circle*{0.15}}\put(1,1){\circle*{0.15}}
\put(2,1){\circle*{0.15}}\put(3,1){\circle*{0.15}}
\put(4,1){\circle*{0.15}}\put(5,1){\circle*{0.15}}
\put(6,1){\circle*{0.15}}\put(7,1){\circle*{0.15}}
\put(8,1){\circle*{0.15}}\put(9,1){\circle*{0.15}}
\put(0,2){\circle*{0.15}}\put(1,2){\circle*{0.15}}
\put(2,2){\circle*{0.15}}\put(3,2){\circle*{0.15}}
\put(4,2){\circle*{0.15}}\put(5,2){\circle*{0.15}}
\put(6,2){\circle*{0.15}}\put(7,2){\circle*{0.15}}
\put(8,2){\circle*{0.15}}\put(9,2){\circle*{0.15}}
\color{red}
\put(8,2){\vector(-4,-1){4}}
\color{blue}
\put(5,2){\vector(3,-1){3}}
\color{green}
\put(3,2){\vector(3,-1){3}}
\color{cyan}
\put(7,2){\vector(-4,-1){4}}
\color{black}
\put(4,1){\line(-1,-1){1}}
\put(3,1){\line(-1,-1){1}}
\put(1,1){\line(1,1){1}}
\put(5,0){\line(4,1){4}}
\end{picture} 
$$
\end{proof}

\begin{remark}
\label{rem:lemma5forn=4}
Lemma~\ref{lem:alpha_cong_nowhere=beta}
also holds for $n=4$, except (up to conjugacy and order) in the case 
$\alpha=(1\ 2\ 3)$, $\beta=(1\ 2)(3\ 4)$.  
One can check this case by case.
\end{remark}

\begin{cora}
\label{cor:fixedpointfreeproduct}
If $\alpha,\beta\in S_n$, $n\ge 4$, and $\alpha$ is  a fixed point free
permutation,
then for some $\sigma\in S_n$, $\alpha^\sigma\beta$ is also fixed point free.
\end{cora}
\begin{proof}
By Lemma~\ref{lem:alpha_cong_nowhere=beta}
and Remark~\ref{rem:lemma5forn=4},
there exists $\sigma$ with $(\alpha^{-1})^{\sigma}(i)\not=\beta(i)$
for $i=1,\dots,n$, and so
$\alpha^{\sigma}\beta(i)\not=i$ for $i=1,\dots,n$, i.e., 
$\alpha^{\sigma}\beta$ has no fixed points.
\end{proof}

\begin{lemma}\label{fixedpointfree}
Let $\alpha, \beta \in S_m$, $m\ge 4$,
with $\alpha$ fixed point free and $\beta\not=e$. 
We may regard $\alpha$ and $\beta$ as elements in 
$S_n$ for any $n> m$ by defining $\alpha(i)=\beta(i)=i$ for any $m<i\leq n$.
Then  
$\eta(\alpha^{S_m}\beta^{S_m})<\eta(\alpha^{S_n}\beta^{S_n})$.
\end{lemma}
\begin{proof}
All elements of $\alpha^{S_m}\beta^{S_m}$, considered as elements of
$S_n$, fix at least $n-m$ points.
Thus it suffices to show that
some element of
$\alpha^{S_n}\beta^{S_n}$ fixes fewer than $m-n$ points.
By replacing $\alpha$ and $\beta$ by conjugates if necessary, by
Corollary ~\ref{cor:fixedpointfreeproduct}, 
we may assume that $\alpha\beta$ does not fix
$1,\dots,m$, and that $\beta$ does not
fix $m$. Set $r=\beta^{-1}(m)$ and $s=\beta^{-1}\alpha^{-1}(m)$.
Now
$\alpha^{(m\ m+1)}\beta(i)=\alpha\beta(i)$
unless 
$i=r$, $i=s$, or
$i=\beta^{-1}(m+1)=\beta^{-1}\alpha^{-1}(m+1)=m+1$.
Because $\alpha$ is fixed point free, 
$\alpha(m)\not=m$, and so we have
\begin{eqnarray*}
\alpha^{(m\ m+1)}\beta(r)
&=&\alpha^{(m\ m+1)}(m)
=(m\ m+1)\alpha(m+1)=m\not=r,\\
\alpha^{(m\ m+1)}\beta(s)
&=&
\alpha^{(m\ m+1)}\beta( \beta^{-1}\alpha^{-1}(m))=\alpha^{( m\ m+1)}    \alpha^{-1}(m)
=m+1\not=s\\
\alpha^{(m\ m+1)}\beta(m+1)
&=&
\alpha^{(m\ m+1)}(m+1)=
(m\ m+1)\alpha(m)=\alpha(m)\not=m+1,
\end{eqnarray*}
and so $\alpha^{(m\ m+1)}\beta$ has $n-m-1$ fixed points and
the result follows.
\end{proof}
\begin{lemma}
\label{producttransposition}
Let $n>3$ be an integer. Then $\eta((1\ 2)^{S_n}(1\ 2)^{S_n})=3$.
\end{lemma}
\begin{proof}
An element of 
$(1\ 2)^{S_n}(1\ 2)^{S_n}$ has the form $(i\ j)(k\ l)$, with
$i,j,k,l\in \{1, \ldots, n\}$ and
 $i\neq j$, $k\neq l$.
Depending on the size of the set $\{i, j\}\cap\{k, l\}$,
$(i\ j)(k\ l)$ is conjugate to one of the following:
$(1\ 2)(1\ 2)=e$,
$(1\ 2)(1\ 3)=(1\ 3\ 2)$, $(1\ 2)(3\ 4)$.
Thus the permutations in  $(1\ 2)^{S_n}(1\ 2)^{S_n}$ are
 the identity, 3-cycles and the product of two disjoint transpositions. By
 Lemma \ref{lemma1} we have then that 
  $\eta((1\ 2)^{S_n}(1\ 2)^{S_n})=3$.
\end{proof}
\begin{remark}
The example given by the previous Lemma
shows that the hypothesis that 
$\alpha$ is fixed point free can not be dropped from 
Lemma \ref{fixedpointfree}.
\end{remark}
\begin{lemma}
\label{lem:n=4eta=1}
If $\alpha,\beta\in S_4\setminus\{e\}$
and $\eta(\alpha^{S_4}\beta^{S_4})=1$, then up to conjugation 
$\{\alpha,\beta\}=\{(1\ 2\ 3),(1\ 2)(3\ 4)\}$.
\end{lemma}
\begin{proof}
This can be checked by hand, or by computer
e.g., MAGMA \cite{magma}.
\end{proof}
\begin{lemma}\label{atleastone}
If $\alpha,\beta\in S_n\setminus\{e\}$ then 
there exists a permutation $\gamma \in S_n$ 
such that $\gamma$ fixes at least one point and 
$\gamma \in \alpha^{S_n}\beta^{S_n}$. 
\end{lemma}
\begin{proof}
Since $\alpha,\beta\in S_n\setminus\{e\}$, by Lemma \ref{lemma1}
we may assume that $\alpha(1)=2$ and $\beta(2)=1$. Thus
$\alpha\beta(1)=1$ and the proof is complete.
\end{proof}

\begin{lemma}
\label{lem:fixedpointfree}
Let  $n\ge 5$ and $\alpha,\beta\in S_n\setminus\{e\}$. If
$\eta(\alpha^{S_n}\beta^{S_n})\le 2$ 
then $\eta(\alpha^{S_n}\beta^{S_n})= 2$, and
at least  one of $\alpha$ and
$\beta$ is a  fixed point free permutation.
\end{lemma}
\begin{proof}
Suppose $\alpha,\beta\in S_n$ both fix a point, and
$\eta(\alpha^{S_5}\beta^{S_n})\le2$.
Because $\alpha$ and $\beta$ fix some element, they can be considered as
elements of $S_{n-1}$.   Lemma~\ref{fixedpointfree} then implies that
$\eta(\alpha^{S_{n-1}}\beta^{S_{n-1}})=1$.
If $n\ge6$, we can assume the result for $n-1$ inductively, and
this is a contradiction.
If $n=5$, 
by Lemma~\ref{lem:n=4eta=1}, up to conjugation,
$\alpha=(1\ 2\ 3)$ and $\beta=(1\ 2)(3\ 4)$.
But in $S_5$, we have
$(1\ 2\ 3)(1\ 2)(3\ 4)=(1\ 3\ 4)(2)$,
$(1\ 2\ 3)(1\ 2)(4\ 5)=(1\ 3)(4\ 5)$,
$(1\ 2\ 3)(1\ 4)(2\ 5)=(1\ 4\ 2\ 5\ 3)$,
which are all in different conjugacy classes, so 
$\eta(\alpha^{S_5}\beta^{S_5})\ge3$, a contradiction. Thus
we may assume that at least one of  $\alpha$ or $\beta$ 
is a fixed point 
free permutation. By Corollary \ref{cor:fixedpointfreeproduct}
and Lemma~\ref{atleastone} we have that
$\eta(\alpha^{S_n}\beta^{S_n})\geq 2$ and the result follows.
\end{proof}

\begin{remark}
Lemma~\ref{lem:n=4eta=1} shows that
the hypothesis $n\ge 5$ is necessary in Lemma~\ref{lem:fixedpointfree}.
\end{remark}

\begin{lemma}
\label{lem:onefixedpoint}
Let $n\ge 7$ and $\alpha, \beta\in S_n\setminus\{e\}$.
If at least one of $\alpha$ and $\beta$ has a cycle of length
at least three, and at least one of $\alpha$, $\beta$
is fixed point free, then for some $\sigma\in S_n$,
$\alpha^\sigma\beta$ has exactly one fixed point.
\end{lemma}
\begin{proof}
Since $\alpha, \beta\not=e$, after conjugation, we may assume
$\alpha(2)=1, \beta(1)=2$.
We have three cases:
(i) $\alpha$ and $\beta$ both  contain cycles of length at least three;
(ii) only $\alpha$ has a cycle length at least three;
(iii) only $\beta$ has a cycle length at least three.
In these cases, illustrated in the diagram below, we may conjugate so that 
(i) $\alpha(3)=2, \beta(2)=4$,
or
(ii) $\alpha(1)=3, \beta(2)=1$,
or
(iii) $\alpha(1)=2, \beta(3)=1$ respectively.
$$
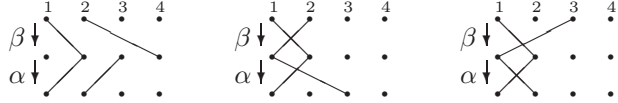
\setlength\unitlength{0.5cm}
\begin{picture}(15,2.5)(-1,0)
\put(-1,0.3){$\alpha$}
\put(-1,1.3){$\beta$}
\put(-0.3,1.9){\vector(0,-1){0.6}}
\put(-0.3,0.9){\vector(0,-1){0.6}}
\put(-0.08,2.2){\tiny 1}
\put(0.9,2.2){\tiny2}
\put(1.9,2.2){\tiny3}
\put(2.9,2.2){\tiny4}
\put(0,0){\circle*{0.15}}\put(1,0){\circle*{0.15}}
\put(2,0){\circle*{0.15}}\put(3,0){\circle*{0.15}}
\put(0,1){\circle*{0.15}}\put(1,1){\circle*{0.15}}
\put(2,1){\circle*{0.15}}\put(3,1){\circle*{0.15}}
\put(0,2){\circle*{0.15}}\put(1,2){\circle*{0.15}}
\put(2,2){\circle*{0.15}}\put(3,2){\circle*{0.15}}
\put(0,2){\line(1,-1){1}}
\put(1,1){\line(-1,-1){1}}
\put(2,1){\line(-1,-1){1}}
\put(1,2){\line(2,-1){2}}
\put(6,0){
\put(-1,0.3){$\alpha$}
\put(-1,1.3){$\beta$}
\put(-0.3,1.9){\vector(0,-1){0.6}}
\put(-0.3,0.9){\vector(0,-1){0.6}}
\put(-0.08,2.2){\tiny 1}
\put(0.9,2.2){\tiny2}
\put(1.9,2.2){\tiny3}
\put(2.9,2.2){\tiny4}
\put(0,0){\circle*{0.15}}\put(1,0){\circle*{0.15}}
\put(2,0){\circle*{0.15}}\put(3,0){\circle*{0.15}}
\put(0,1){\circle*{0.15}}\put(1,1){\circle*{0.15}}
\put(2,1){\circle*{0.15}}\put(3,1){\circle*{0.15}}
\put(0,2){\circle*{0.15}}\put(1,2){\circle*{0.15}}
\put(2,2){\circle*{0.15}}\put(3,2){\circle*{0.15}}
\put(0,2){\line(1,-1){1}}
\put(1,1){\line(-1,-1){1}}
\put(0,1){\line(2,-1){2}}
\put(0,1){\line(1,1){1}}
}
\put(12,0){
\put(-1,0.3){$\alpha$}
\put(-1,1.3){$\beta$}
\put(-0.3,1.9){\vector(0,-1){0.6}}
\put(-0.3,0.9){\vector(0,-1){0.6}}
\put(-0.08,2.2){\tiny 1}
\put(0.9,2.2){\tiny2}
\put(1.9,2.2){\tiny3}
\put(2.9,2.2){\tiny4}
\put(0,0){\circle*{0.15}}\put(1,0){\circle*{0.15}}
\put(2,0){\circle*{0.15}}\put(3,0){\circle*{0.15}}
\put(0,1){\circle*{0.15}}\put(1,1){\circle*{0.15}}
\put(2,1){\circle*{0.15}}\put(3,1){\circle*{0.15}}
\put(0,2){\circle*{0.15}}\put(1,2){\circle*{0.15}}
\put(2,2){\circle*{0.15}}\put(3,2){\circle*{0.15}}
\put(0,2){\line(1,-1){1}}
\put(1,1){\line(-1,-1){1}}
\put(0,1){\line(1,-1){1}}
\put(0,1){\line(2,1){2}}
}
\end{picture} 
$$
In all cases, $\alpha\beta(1)=1$ and $\alpha\beta(2)\not=2$.
Now we proceed with the same inductive construction as in
the proof of Lemma~\ref{lem:alpha_cong_nowhere=beta}, starting
at the step $k=3$, since $\alpha$
has already been conjugated so that $\alpha\beta(1)=1$ and 
$\alpha\beta(2)\not=2$. 
Case 2 of the procedure never occurs, since by assumption
one of $\alpha$ or $\beta$ is fixed point free.
When case 3 occurs, we must conjugate by $(\alpha(k)\ t)$
for some $t\notin S$, with $S$
as in Lemma~\ref{lem:alpha_cong_nowhere=beta}.
Since $|S|\le4$,
provided $n\ge 7$, we can pick $t\not\in\{1,2\}\cup S$.
Then the property $\alpha\beta(1)=1$ will be unaltered by
replacing $\alpha$ by $\alpha^{(\alpha(k)\ t)}$.
\end{proof}

\begin{remark}
\label{rem:Lemma12in_casen=6}
Lemma~\ref{lem:onefixedpoint}
fails when $n=6$ and $\alpha=(1\ 2\ 3)(4\ 5\ 6)$, 
$\beta=(1\ 2)(3\ 4)(5\ 6)$.
Up to conjugation and change of order, this is the only pair
of $\alpha$, $\beta$ in $S_6$ which satisfy the
hypothesis, but fail the conclusion of
Lemma~\ref{lem:onefixedpoint}.   The possible cases
may be checked by hand or computer.
\end{remark}

\begin{lemma}
\label{lem:atleasttwofixed}
Let  $\alpha, \beta\in S_n\setminus \{e\}$ be permutations.
Assume that at least one of  $\alpha$ and $\beta$ is fixed point 
free.  If either (i)
both $\alpha, \beta$ contain a cycle of length at least three,
(ii) 
$\alpha, \beta$ have at least $4$ non-fixed points, 
or (iii) 
both $\alpha, \beta$ contain a transposition,
then
there exists $\sigma\in S_n$ such that
$\alpha^\sigma\beta$ has at least two fixed points.
\end{lemma}
\begin{proof}
We may assume, after taking conjugates, that
 $\alpha, \beta$ act on $1$, $2$, $3$, $4$ as
in the following diagram, where
lines indicate conditions on the mapping, 
e.g., in all cases $\alpha(1)=2$;
if a line is not given, then no requirement
is made.
$$
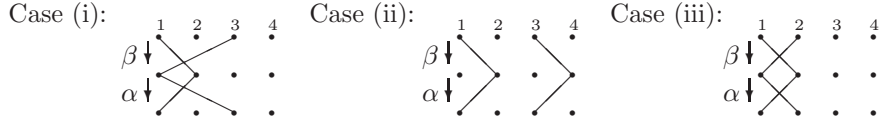
\setlength\unitlength{0.5cm}
\begin{picture}(17,3)(-1,0)
\put(-4,2.4){Case (i):}
\put(-1,0.3){$\alpha$}
\put(-1,1.3){$\beta$}
\put(-0.3,1.9){\vector(0,-1){0.6}}
\put(-0.3,0.9){\vector(0,-1){0.6}}
\put(-0.08,2.2){\tiny1}
\put(0.9,2.2){\tiny2}
\put(1.9,2.2){\tiny3}
\put(2.9,2.2){\tiny4}
\put(0,0){\circle*{0.15}}\put(1,0){\circle*{0.15}}
\put(2,0){\circle*{0.15}}\put(3,0){\circle*{0.15}}
\put(0,1){\circle*{0.15}}\put(1,1){\circle*{0.15}}
\put(2,1){\circle*{0.15}}\put(3,1){\circle*{0.15}}
\put(0,2){\circle*{0.15}}\put(1,2){\circle*{0.15}}
\put(2,2){\circle*{0.15}}\put(3,2){\circle*{0.15}}
\put(0,2){\line(1,-1){1}}
\put(1,1){\line(-1,-1){1}}
\put(0,1){\line(2,-1){2}}
\put(0,1){\line(2,1){2}}
\put(8,0){
\put(-4,2.4){Case (ii):}
\put(-1,0.3){$\alpha$}
\put(-1,1.3){$\beta$}
\put(-0.3,1.9){\vector(0,-1){0.6}}
\put(-0.3,0.9){\vector(0,-1){0.6}}
\put(-0.08,2.2){\tiny1}
\put(0.9,2.2){\tiny2}
\put(1.9,2.2){\tiny3}
\put(2.9,2.2){\tiny4}
\put(0,0){\circle*{0.15}}\put(1,0){\circle*{0.15}}
\put(2,0){\circle*{0.15}}\put(3,0){\circle*{0.15}}
\put(0,1){\circle*{0.15}}\put(1,1){\circle*{0.15}}
\put(2,1){\circle*{0.15}}\put(3,1){\circle*{0.15}}
\put(0,2){\circle*{0.15}}\put(1,2){\circle*{0.15}}
\put(2,2){\circle*{0.15}}\put(3,2){\circle*{0.15}}
\put(0,2){\line(1,-1){1}}
\put(1,1){\line(-1,-1){1}}
\put(2,2){\line(1,-1){1}}
\put(3,1){\line(-1,-1){1}}
}
\put(16,0){
\put(-4,2.4){Case (iii):}
\put(-1,0.3){$\alpha$}
\put(-1,1.3){$\beta$}
\put(-0.3,1.9){\vector(0,-1){0.6}}
\put(-0.3,0.9){\vector(0,-1){0.6}}
\put(-0.08,2.2){\tiny 1}
\put(0.9,2.2){\tiny2}
\put(1.9,2.2){\tiny3}
\put(2.9,2.2){\tiny4}
\put(0,0){\circle*{0.15}}\put(1,0){\circle*{0.15}}
\put(2,0){\circle*{0.15}}\put(3,0){\circle*{0.15}}
\put(0,1){\circle*{0.15}}\put(1,1){\circle*{0.15}}
\put(2,1){\circle*{0.15}}\put(3,1){\circle*{0.15}}
\put(0,2){\circle*{0.15}}\put(1,2){\circle*{0.15}}
\put(2,2){\circle*{0.15}}\put(3,2){\circle*{0.15}}
\put(0,2){\line(1,-1){1}}
\put(1,1){\line(-1,-1){1}}
\put(0,1){\line(1,-1){1}}
\put(0,1){\line(1,1){1}}
}
\end{picture} 
$$
In cases (i) and (ii), $1$ and $3$ are fixed,
and in case (iii), $1$ and $2$ are fixed.
\end{proof}

\begin{cora}
\label{cor:caseswheneta=2}
If $n\ge 6$, $\alpha, \beta\in S_n$ and $\eta(\alpha^{S_n}\beta^{S_n})=2$, then up to change
of order of $\alpha$, $\beta$, $\alpha$ is fixed point free, and one
of the following holds:
\begin{enumerate}
\item[(i)]    $\alpha$ contains a cycle of length at least three,
and $\beta$ is a transposition.
\item[(ii)]  $\alpha$ is a product of disjoint transpositions,
and $\beta$ is a three cycle.
\item[(iii)] Both $\alpha$ and $\beta$ are products of
disjoint transpositions.  
\end{enumerate}
\end{cora}
\begin{proof}
By Lemma~\ref{lem:fixedpointfree}, one of $\alpha$ and 
 $\beta$ is
fixed point free, and by Lemma~\ref{interchange} without loss of
generality, we may assume that $\alpha$ is fixed point free.
By Corollary~\ref{cor:fixedpointfreeproduct} 
and Lemmas~\ref{lem:onefixedpoint}, \ref{lem:atleasttwofixed}
and Remark~\ref{rem:Lemma12in_casen=6},
it follows that unless we are in cases (i), (ii), (iii),
or in case $n=6$ and $\{\alpha^{S_6},\beta^{S_6}\}=\{((1\ 2\ 3)(4\ 5\ 6))^{S_6},
((1\ 2)(3\ 4)(5\ 6))^{S_6}\}$, then
$\alpha^{S_n}\beta^{S_n}$ contains elements with no fixed points,
with exactly one fixed point, and with at least two fixed points.
These are in different conjugacy classes from each other, so the result follows,
except for the case 
$n=6$ and $\{\alpha^{S_6},\beta^{S_6}\}=\{((1\ 2\ 3)(4\ 5\ 6))^{S_6},
((1\ 2)(3\ 4)(5\ 6))^{S_6}\}$.  In the remaining case,
we can explicitly see that $\eta(\alpha^{S_6},\beta^{S_6})\ge3$, since
$(1\ 2)(3\ 4)(5\ 6)(1\ 2\ 3)(4\ 5\ 6)=(2\ 4\ 6\ 3),$
$(1\ 4)(2\ 5)(3\ 6)(1\ 2\ 3)(4\ 5\ 6)
=(1\ 5\ 3\ 4\ 2\ 6)$ and
$(1\ 4)(2\ 6)(3\ 5)(1\ 2\ 3)(4\ 5\ 6)= (1\ 6)(2\ 5)(3\ 4)
$,
and so the result also holds for this case.
\end{proof}

\begin{proof}[Proof of Theorem A]
By Lemma~\ref{lem:fixedpointfree} 
the minimal value of $\eta$ when $\alpha, \beta$
are non trivial is at least $2$, and by Lemma~\ref{producttransposition}, 
it is at most $3$.
Corollary~\ref{cor:caseswheneta=2} 
gives three cases when the minimal value is $2$.

Case (i): $\beta$ is a transposition, and
$\alpha$ is fixed point free and contains a cycle of length at least three.
Note that  $(1\ 2\ \cdots\ r)(1\ 2)=(2)(1\ 3\ \cdots\ r)$
and that for $s>r$,
$(1\ 2\ \cdots\ r)(r+1\ r+2\ \cdots s)(r\ r+1)=
(1\ 2\ \cdots r\ r+2\ r+3\ \cdots s\ r+1)$.
This implies that if $\Type(\alpha)=\{a_1,\cdots,a_k\}$,
then for $1\le i\not=j\le k$,
 $\alpha^{S_n}\beta^{S_n}$ contains elements with
cycle types
$\Type(\alpha)\setminus\{a_i\}\cup\{1,a_i-1\}$
and
$\Type(\alpha)\setminus\{a_i,a_j\}\cup\{a_i+a_j\}$,
(where these are all operations on multisets, not sets). 
If $r>3$, observe that
$(1\ 2\ \cdots\ r)(1\ 3)=(1\ 4\ \cdots\ r)(2 \ 3)$.
Thus $\alpha^{S_n}\beta^{S_n}$ contains  an
element with cycle type
$\Type(\alpha)\setminus\{a_i\}\cup\{a_i-2, 2\}$ if $a_i>3$ 
for some $i$. 
So, if $\eta(\alpha^{S_n}\beta^{S_n})=2$, we must have
that $a_i=a_j=3$ for all $j$ and so we must be in
case (ii) of 
the theorem.

Case (ii): This is the second possibility of case (i) of 
the theorem.

Case (iii): Suppose $\beta$ consists of at least two disjoint
transpositions.  Suppose $\alpha=(1\ 2)\alpha_1$,
and $\beta=(1\ 2)\beta_1$ where $\alpha_1, \beta_1$ fix $1$ and $2$.
As elements of $S_{n-2}$, $\alpha_1$ is fixed point free,
and $\beta_1$ contains a transposition, so by
Corollary~\ref{cor:fixedpointfreeproduct} and
Lemma~\ref{lem:atleasttwofixed} $\alpha_1^{S_{n-1}}\beta_1^{S_{n-1}}$
contain elements with no fixed points, and elements with
at least two fixed points.
Composing these elements with
$(1\ 2)$ gives elements in $\alpha^{S_n}\beta^{S_n}$
with exactly $2$ fixed points, and with at least $4$ fixed points.
On the other hand, by Corollary~\ref{cor:fixedpointfreeproduct},
$\alpha^{S_n}\beta^{S_n}$ contains elements with no fixed points.
Since all of these elements are in different conjugacy classes,
we have $\eta(\alpha^{S_n}\beta^{S_n})\ge 3$.
So, for $\eta=2$, we must be in the first possibility of case (i)
of the theorem.

Finally, it is easy to check that in the cases of the theorem,
we do indeed have $\eta(\alpha^{S_n}\beta^{S_n})=2$.
\end{proof}

\end{section}

\end{document}